\documentclass[11pt]{amsart}

\usepackage{amsmath,amsthm,amssymb,dsfont,enumerate,tikz,pgfplots,hyperref,float}
\usepackage{setspace}
\usepackage[margin=1.25in]{geometry}

\newtheorem{theorem}{Theorem}[section]
\newtheorem{proposition}[theorem]{Proposition}

\theoremstyle{definition}
\newtheorem{definition}[theorem]{Definition}
\newtheorem{algorithm}[theorem]{Algorithm}

\newtheorem{remark}[theorem]{Remark}

\numberwithin{equation}{section}
\numberwithin{figure}{section}

\renewcommand{\subset}{\subseteq}
\renewcommand{\hat}{\widehat}
\renewcommand{\tilde}{\widetilde}
\renewcommand{\epsilon}{\varepsilon}
\def\diam{\text{diam}}

\def\supp{\text{supp}}

\def\<{\langle}
\def\>{\rangle}
\def\({\Big(}
\def\){\Big)}
\def\C{\mathbb{C}}
\def\F{\mathcal{F}}

\def\P{\mathcal{P}}

\def\R{\mathbb{R}}
\def\S{\mathcal{S}}

\def\W{\mathcal{W}}
\def\Z{\mathbb{Z}}

\title{Analysis of time-frequency scattering transforms}
\author{Wojciech Czaja}
\address{Department of Mathematics, University of Maryland, College Park, MD 20742}
\email{wojtek@math.umd.edu}
\author{Weilin Li}
\address{Department of Mathematics, University of Maryland, College Park, MD 20742}
\email{wl298@math.umd.edu}
\date{}

\begin{document}
	
\begin{abstract}
	In this paper we address the problem of constructing a feature extractor which combines Mallat's scattering transform framework with time-frequency (Gabor) representations. To do this, we introduce a class of frames, called \emph{uniform covering frames}, which includes a variety of semi-discrete Gabor systems. Incorporating a uniform covering frame with a neural network structure yields the \emph{Fourier scattering transform} $\S_\F$ and the \emph{truncated Fourier scattering transform}. We prove that $\S_\F$ propagates energy along frequency decreasing paths and its energy decays exponentially as a function of the depth. These quantitative estimates are fundamental in showing that $\S_\F$ satisfies the typical scattering transform properties, and in controlling the information loss due to width and depth truncation. We introduce the \emph{fast Fourier scattering transform} algorithm, and illustrate the algorithm's performance. The time-frequency covering techniques developed in this paper are flexible and give insight into the analysis of scattering transforms.
\end{abstract}

\maketitle

\section{Introduction}

Introduced by LeCun \cite{le1990handwritten}, a \emph{convolutional neural network} is a composition of a finite number of transformations, where each transformation is one of three types: a convolution against a filter bank, a non-linearity, and an averaging. Convolutional neural networks approximate functions through an adaptive and iterative learning process and have been extremely successful for classifying data \cite{le1990handwritten, hinton2012deep, krizhevsky2012imagenet}. Since they have complex architectures and their parameters are learned through ``black-box" optimization schemes, training is computationally expensive and there is no widely accepted rigorous theory that explains their remarkable success.  

Recently, Mallat \cite{mallat2012group} provided an intriguing example of a predetermined convolutional neural network with formal mathematical guarantees. His \emph{windowed scattering transform} $\S_\W$ propagates the input information through multiple iterations of the wavelet transform and the complex modulus, and finishes the process with a local averaging. It is typically used as a \emph{feature extractor}, which is a transformation that organizes the input data into a particular form, while simultaneously discarding irrelevant information.  When combined with standard classifiers, the windowed scattering transform has achieved state-of-the-art results for several classification problems \cite{bruna2013invariant, sifre2013rotation, hirn2015quantum}.  

While Mallat's results are impressive, there are several reasons to consider an alternative case where a non-wavelet frame is used for scattering.

\begin{enumerate}[(a)]\itemsep+.5em
	
	\item 
	Since neural networks were originally inspired by the structure of the brain, it makes sense to mimic the visual system of mammals when designing a feature extractor for image classification. The ground-breaking work of Daugman \cite{daugman1985uncertainty, daugman1988complete} demonstrated that simple cells in the mammalian visual cortex are modeled by modulations of a fixed 2-dimensional Gaussian. In other words, this is a Gabor system with a Gaussian window. We remark that modern neural networks also incorporate ideas that are not strictly biologically motivated. 
	
	\item
	The authors of \cite{lee2008sparse} observed that the learned filters (the experimentally ``optimal" filters) in Hinton's algorithm for learning \emph{deep belief networks} \cite{hinton2006fast} are localized, oriented, band-pass filters, which resemble Gabor functions. Strictly speaking, this set of functions is not a Gabor system since it is not derived from a single generating function, but it is not a wavelet system either.
	
	\item
	The use of Gabor frames for classification is not unprecedented, since the short-time Fourier transform with Gaussian window has been used as a feature extractor for various image classification problems \cite{hamamoto1998gabor, kong2003palmprint, arivazhagan2006texture}. These papers predated Mallat's work on scattering transforms and did not use Gabor functions in a multi-layer decomposition. To the best of our knowledge, we are not aware of any prior work that combines neural networks with Gabor functions. 
\end{enumerate} 

We address the situation where a Gabor frame is used for scattering. In Section \ref{section definitions}, we introduce a new class of frames, called \emph{uniform covering frames}, and this class is a natural generalization of certain types of Gabor frames. In fact, no wavelet frame is a uniform covering frame, so our situation is completely different from that in Mallat \cite{mallat2012group}. We combine uniform covering frames with neural networks to obtain the \emph{Fourier scattering transform} $\S_\F$ and the \emph{truncated Fourier scattering transform} $\S_\F[M,K]$, where the parameters $M$ and $K$ control the width and depth the network, respectively. 

In Section \ref{section FST}, we concentrate on the theoretical analysis of $\S_\F$. Since both $\S_\F$ and $\S_\W$ share the same network structure, it is natural to ask whether they share the same broad mathematical properties. We establish an exponential decay of energy estimate for $\S_\F$, Proposition \ref{prop decay}, and it is unclear whether $\S_\W$ satisfies this property. Theorem \ref{thm1} shows that $\S_\F$ conserves energy, is non-expansive, and contracts sufficiently small translations. It also shows that the transformation contracts sufficiently small diffeomorphisms assuming additional regularity on the input data.

In Section \ref{section truncation}, we establish analogous estimates for $\S_\F[M,K]$. The main difficulty is ensuring that the truncation has trivial kernel. We show that the largest coefficients of $\S_\F$ are concentrated in the frequency decreasing paths, Proposition \ref{prop path}, and it is unknown if $\S_\W$ satisfies this property. By using this quantitative control over the coefficients of $\S_\F$, we prove Theorem \ref{thm2}, which shows that $\S_\F[M,K]$ is an effective feature extractor for appropriate choices of the parameters $M$ and $K$.  

In Section \ref{section motivation}, we introduce the \emph{fast Fourier scattering transform} algorithm, which computes $\S_\F[M,K]$. We use this algorithm to compute the largest Fourier scattering coefficients of a model image. Our experiment demonstrates that, while the first-order coefficients identifies the edges in the image, the second-order ones extract global and subtle oscillatory features. We also qualitatively compare the features generated by $\S_\F$ and $\S_\W$. 

In Section \ref{section discussion}, we conclude this paper with a detailed comparison of our theoretical results with those from the literature. In particular, our proofs use covering arguments that have not been previously applied to the analysis of scattering transforms.

\section{Definitions}
\label{section definitions}

Mathematically, a feature extractor is an operator, $S\colon X\to Y$, where $X$ and $Y$ are metric spaces. We primarily work with the data space $X=L^2(\R^d)$, the space of Lebesgue measurable functions that are square integrable, which provides an accurate model for audio and image data. 

For several results, we require additional regularity on the input data. The Fourier transform of a Schwartz function $f$ is $\hat f(\xi)=\int_{\R^d} f(x)e^{-2\pi ix\cdot \xi}\ dx$, and this definition has a unique extension to $f\in L^2(\R^d)$ by density, see \cite{benedetto1996harmonic}. For $R>0$ and $x\in\R^d$, let $Q_R(x)=\{y\in\R^d\colon |y-x|_\infty <R\}$ be the open cube of size length $2R$ centered at $x$. We say that a function $f\in L^2$ is $(\epsilon,R)$ band-limited for some $\epsilon\in[0,1)$ and $R>0$, if $\|\hat f\|_{L^2(Q_R(0))}\geq (1-\epsilon)\|f\|_{L^2}$. Of course, if it is possible to find $R$ sufficiently large such that $\epsilon=0$, then $f$ is band-limited. This assumption is realistic, since it has been observed that natural images are essentially band-limited \cite{pennebaker1992jpeg}.

We primarily work with the feature space $Y = L^2(\R^d;\ell^2(\Z))$, the set of sequences $\{f_m\colon m\in\Z\}$ such that all $f_m\colon \R^d\to\C$ are Lebesgue measurable and 
\[
\|\{f_m\}\|_{L^2\ell^2}^2=\int_{\R^d}\sum_{m\in\Z} |f_m(x)|^2\ dx<\infty.
\]

In order to improve classification rates, an effective feature extractor $S$ contracts distances between points belonging to the same class, and expands distances between points belonging to different classes; feature extractors that trivially contract or expand all data points are ineffective. For this reason, we want $S$ to be bounded above and below. Otherwise, we can find sequences $\{f_n\},\{g_n\}\subset X$ of unit norm vectors such that $\|Sf_n\|_Y\to 0$ and $\|Sg_n\|_Y\to\infty$. 

Symmetries and invariants play an important role in feature extraction. For example, a small translation or perturbation of an image does not change its classification. More specifically, for $y\in \R^d$, let $|y|$ be its Euclidean norm and let $|y|_\infty$ be its sup norm. Let $T_y$ be the translation operator 
\[
T_yf(x)=f(x-y).
\]
Let $C^k(\R^d;\R^d)$ be the space of $k$-times continuous differentiable functions from $\R^d$ to $\R^d$ equipped with its usual norm $\|\cdot\|_{C^k}$. For $\tau\in C^1(\R^d;\R^d)$, let $T_\tau$ be the additive diffeomorphism 
\[
T_\tau f(x)=f(x-\tau(x)).
\]
In order to demonstrate that $S$ is an effective feature extractor, we would like to obtain finite upper bounds on $\|S(T_yf)-Sf\|_Y$ and $\|S(T_\tau f)-Sf\|_Y$ in terms of $|y|$, $\|\tau\|_{C^1}$, and $\|f\|_X$. See \cite{benedetto2010integration} for basic real analysis facts.

Mallat's windowed scattering transform \cite{mallat2012group} satisfies variants of these properties, and we carefully discuss his results in Section \ref{section discussion}. He constructed the windowed scattering transform by combining a specific wavelet frame with convolutional neural networks. Let $J$ be an integer and let $G$ be a finite group of rotations on $\R^d$ together with reflection about the origin. Consider the wavelet frame,
\[
\W=\{\varphi_{2^J}\}\cup \{\psi_{2^j,r}\colon j>-J,\ r\in G\},
\]
where $\varphi_{2^J}(x)=2^{-dJ}\varphi(2^{-J}x)$ is the wavelet corresponding to the coarsest scale $2^J$, and $\psi_{2^j,r}(x)=2^{dj}\psi(2^jr^{-1}x)$ is a detail wavelet of scale $2^{-j}$ and localization $r$. Here, we have followed Mallat's notation where the dilations of $\varphi$ and $\psi$ are inversely related. The index set of $\W$ is the countably infinite set
\[
\Lambda=\{(2^j,r)\colon j>-J,\ r\in G\}.
\]
The network structure is combined with the wavelet frame by creating a tree from the index set $\Lambda$ and associating each element of the tree with a corresponding operator. Indeed, let $\Lambda^0=\emptyset$, and for integers $k\geq 1$, let
\[
\Lambda^k=\underbrace{\Lambda\times\Lambda\times\cdots\times\Lambda}_{k-\text{times}}.
\]
Then, each $\lambda\in\Lambda^k$ is associated with the \emph{scattering propagator} $U[\lambda]$, formally defined as 
\[
U[\lambda]f
=\begin{cases}
\ f &\text{if } \lambda\in\Lambda^0, \\
\ |f*\psi_\lambda| &\text{if } \lambda\in\Lambda, \\
\ U[\lambda_k] U[\lambda_{k-1}] \cdots U[\lambda_1] f &\text{if } \lambda=(\lambda_1, \lambda_2, \dots,\lambda_k)\in\Lambda^k.
\end{cases}
\]
Strictly speaking, it does not make sense to write $\lambda\in\Lambda^0=\emptyset$, but we use this convention for convenience, see \cite{mallat2012group}. The \emph{windowed scattering transform} $\S_\W$ is formally defined as
\[
\S_\W(f)
=\{U[\lambda]f*\varphi_{2^J}\colon \lambda\in\Lambda^k,\ k=0,1,\dots\}. 
\]
We importantly mention that even though $\S_\W$ is defined using the wavelet transform, which is unitary, $\S_\W$ is not invertible due to the loss of the phase factor in each layer. This is disadvantageous for certain applications such as compression, but the empirical results \cite{bruna2013invariant, sifre2013rotation, hirn2015quantum} suggest that this property is advantageous for classification problems. 

Mallat's method for combining wavelets and neural networks is flexible, and we use his idea to combine neural networks with time-frequency representations called Gabor frames \cite{benedetto1993gabor}. The essential support of a Lebesgue measurable function $f$, denoted $\supp(f)$, is the complement of the largest open set where $f=0$ almost everywhere.

\begin{definition}
	Let $\P$ be a countably infinite index set. A \emph{uniform covering frame} is a sequence of functions,
	\[
	\F=\{f_0\}\cup\{f_p\colon p\in\P\},
	\]
	satisfying the following assumptions:
	\begin{enumerate}[(a)]\itemsep+.5em
		\item
		\emph{Assumptions on $f_0$ and $f_p$}. Let $f_0\in L^1(\R^d)\cap L^2(\R^d)\cap C^1(\R^d)$ such that $\hat{f_0}$ is supported in a neighborhood of the origin and $|\hat{f_0}(0)|=1$. For each $p\in\P$, let $f_p\in L^1(\R^d)\cap L^2(\R^d)$ such that $\supp(\hat{f_p})$ is compact and connected. 
		\item
		\emph{Uniform covering property}. For any $R>0$, there exists an integer $N>0$ such that for each $p\in\P$, the set $\supp(\hat{f_p})$ can be covered by $N$ cubes of side length $2R$.
		\item
		\emph{Frame condition}. Assume that for all $\xi\in\R^d$,
		\begin{equation}
		\label{eq frame} 
		|\hat{f_0}(\xi)|^2+\sum_{p\in\P} |\hat{f_p}(\xi)|^2=1.
		\end{equation}
		This implies $\F$ is a semi-discrete Parseval frame for $L^2(\R^d)$: For all $f\in L^2(\R^d)$,
		\[
		\|f*f_0\|_{L^2}^2+\sum_{p\in\P} \|f*f_p\|_{L^2}^2 = \|f\|_{L^2}^2.
		\]
	\end{enumerate}
\end{definition}

\begin{remark}
	\label{remark cover}
	The uniform covering property is the key ingredient to our results, and we have several comments on this assumption. 
	
	\begin{enumerate}[(a)]\itemsep+.5em
		\item
		The uniform covering property is, both, a size and a shape constraint on the sets $\{\supp(\hat{f_p})\colon p\in\P\}$. It is a size constraint because it implies $\sup_{p\in\P}|\supp(\hat{f_p})|<\infty$, where $|S|$ denotes the Lebesgue measure of the set $S$. The uniform covering property is also a shape constraint because the number of cubes of a fixed side length required to cover the unit cube is much less than the number required to cover an elongated rectangular prism of unit volume.
		
		\item		
		Since a wavelet frame is partially generated by dilations of a single function, the support of each wavelet varies according to the dilation. Hence, no wavelet frame can satisfy the uniform covering property, and in turn, no wavelet frame can be a uniform covering frame. Examples of wavelet frames include standard wavelets \cite{daubechies1992ten, mallat1999wavelet}, curvelets \cite{candes2004new}, shearlets \cite{guo2006sparse, guo2007optimally,czaja2014anisotropic}, composite wavelets \cite{guo2006wavelets}, $\alpha$-molecules \cite{grohs2015alpha}, and Mallat's scattering wavelets \cite{mallat2012group}.  
		
		\item
		The assumption that $\supp(\hat{f_p})$ is connected is only used to prove Proposition \ref{prop tiling} and the results in Section \ref{section truncation}, when we truncate $\S_\F$. The proposition provides a natural way of thinking about the index set $\P$. The connectedness assumption is used to preclude certain pathological behavior such as $\supp(\hat{f_p})$ having two connected components, where one component is near the origin and the other is far from the origin. 
	\end{enumerate}
\end{remark}

We now explain why uniform covering frames are similar to Gabor frames. A Gabor frame covers the frequency space uniformly by translating a fixed set, while a uniform covering frame covers the frequency domain by sets of approximately equal size and shape. Hence, a uniform covering frame is similar to the time-frequency approach. In contrast to these approaches, a wavelet frame covers the frequency space non-uniformly by dilating a fixed set. 

Given their similarities, it is not surprising that a variety of Gabor frames are uniform covering frames, as shown in the following proposition. It is possible to construct relevant and useful non-Gabor uniform covering frames; in the companion paper \cite{czaja2016uniform}, we construct uniform covering frames that are partially generated by rotations to obtain a rotationally invariant operator, which we call the \emph{rotational Fourier scattering transform}. These are related to previously constructed directional time-frequency representations, see \cite{grafakos2008gabor, czaja2016discrete}.

\begin{proposition}
	\label{prop gabor}
	Let $g\in L^1(\R^d)\cap L^2(\R^d)\cap C^1(\R^d)$ be such that $\supp(\hat g)$ is compact and connected, $|\hat g(0)|=1$, and $\sum_{m\in\Z^d} |\hat g(\xi-m)|^2=1$ for all $\xi\in\R^d$. Let $A\colon \R^d\to\R^d$ be an invertible linear transformation, $f_0(x)=|\det A|\ g(Ax)$, $\P=A^t\Z^d\setminus\{0\}$, and $f_p(x)=e^{2\pi ip\cdot x}f_0(x)$ for each $p\in\P$. Then, $\F=\{f_0\}\cup\{f_p\colon p\in\P\}$ is a Gabor frame, as well as a uniform covering frame. 
\end{proposition}

\begin{proof}
	Since $\supp(\hat{f_p})$ is a translation of the connected and compact set $\supp(\hat{f_0})$, the uniform covering property automatically holds. Let $A^{-t}=(A^{-1})^t$. For all $\xi\in\R^d$, 
	\begin{align*}
	|\hat{f_0}(\xi)|^2+\sum_{p\in\P} |\hat{f_p}(\xi)|^2
	&=|\hat g(A^{-t}\xi)|^2+\sum_{p\in\P} |\hat g(A^{-t}(\xi-p))|^2 \\
	&=\sum_{m\in\Z^d} |\hat g(A^{-t}\xi-m)|^2
	=1.
	\end{align*}
\end{proof}

Figure \ref{fig 2} illustrates the differences between Mallat's wavelet frame $\W$ \cite{mallat2012group}, and the Gabor frame that we just presented.

\begin{figure}[h]
	\begin{center}
	\begin{tikzpicture}[scale=.3]
	\fill [fill,shading=radial,color=gray] (5.6,5.6) circle [radius=4];
	\draw 
	(0,0) circle [radius=2]
	(0,0) circle [radius=4]
	(0,0) circle [radius=8]
	(-10,0)--(10,0)
	(0,10)--(0,-10)
	(7.1,7.1)--(-7.1,-7.1)
	(-7.1,7.1)--(7.1,-7.1);
	\draw [fill] 
	(2,0) circle [radius=.2]
	(-2,0) circle [radius=.2]
	(0,2) circle [radius=.2]
	(0,-2) circle [radius=.2]
	(1.4,1.4) circle [radius=.2]
	(-1.4,-1.4) circle [radius=.2]
	(-1.4,1.4) circle [radius=.2]
	(1.4,-1.4) circle [radius=.2]
	(4,0) circle [radius=.2]
	(-4,0) circle [radius=.2]
	(0,4) circle [radius=.2]
	(0,-4) circle [radius=.2]
	(2.8,2.8) circle [radius=.2]
	(-2.8,-2.8) circle [radius=.2]
	(-2.8,2.8) circle [radius=.2]
	(2.8,-2.8) circle [radius=.2]
	(8,0) circle [radius=.2]
	(-8,0) circle [radius=.2]
	(0,8) circle [radius=.2]
	(0,-8) circle [radius=.2]
	(5.6,5.6) circle [radius=.2]
	(-5.6,-5.6) circle [radius=.2]
	(-5.6,5.6) circle [radius=.2]
	(5.6,-5.6) circle [radius=.2];
	\end{tikzpicture} 
	\begin{tikzpicture}[scale=0.6]
	\fill [fill,color=black!15!white] (1.2,1.2) rectangle (2.8,2.8);
	\draw
	(-1,-1) rectangle (1,1)
	(-2,-2) rectangle (2,2)
	(-3,-3) rectangle (3,3);
	\draw
	(-5,0)--(5,0)
	(0,5)--(0,-5);
	\foreach \x in {-3,...,3}
	\draw[fill] (\x,-3) circle [radius=0.1];
	\foreach \x in {-3,...,3}
	\draw[fill] (\x,-2) circle [radius=0.1];
	\foreach \x in {-3,...,3}
	\draw[fill] (\x,-1) circle [radius=0.1];
	\foreach \x in {-3,...,-1}
	\draw[fill] (\x,0) circle [radius=0.1];
	\foreach \x in {1,...,3}
	\draw[fill] (\x,0) circle [radius=0.1];
	\foreach \x in {-3,...,3}
	\draw[fill] (\x,1) circle [radius=0.1];
	\foreach \x in {-3,...,3}
	\draw[fill] (\x,2) circle [radius=0.1];
	\foreach \x in {-3,...,3}
	\draw[fill] (\x,3) circle [radius=0.1];
	\end{tikzpicture}
	\end{center}
	\caption{Left: Let $G$ be the group of rotations by angle $2\pi/8$. The black dots are elements of $\Lambda$, for the first three dyadic scales. The shaded gray region is the effective support of $(\psi_{2^{-J+3},2\pi/8})^\wedge$. Right: Let $A$ be the identity transformation. The black dots are elements of $\P$ for the first three uniform Fourier scales. The shaded region is the support of $(f_{2,2})^\wedge$.}
	
	\label{fig 2}
\end{figure}
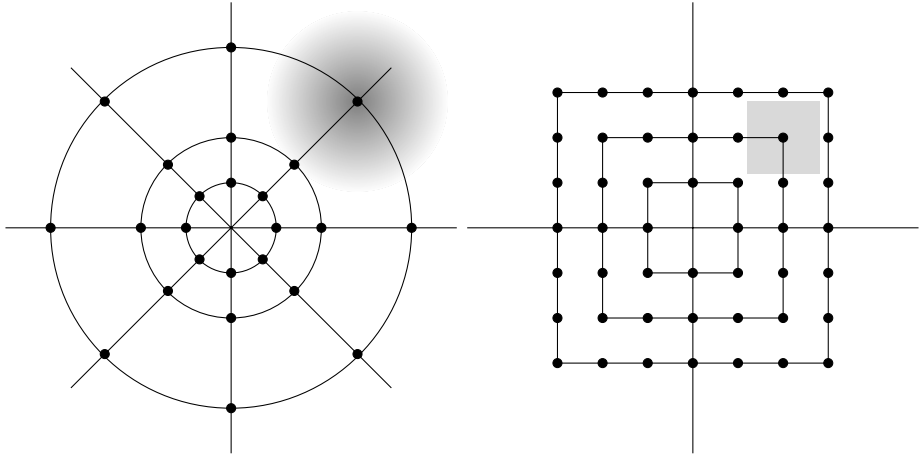

Having established the existence of a large class of uniform covering frames, we return our attention to incorporating the network structure. Slightly abusing notation, we associate $p\in\P^k$ with the scattering propagator $U[p]$, defined as 
\[
U[p]f
=\begin{cases}
\ f &\text{if } p\in\P^0, \\
\ |f*f_p| &\text{if } p\in\P, \\
\ U[p_k] U[p_{k-1}] \cdots U[p_1] f &\text{if } p=(p_1,p_2,\dots,p_k)\in\P^k.
\end{cases}
\]

\begin{definition}
	The \emph{Fourier scattering transform}, $\S_\F$, is the vector-valued operator
	\[
	\S_\F(f)=\{U[p]f*f_0\colon p\in\P^k,\ k=0,1,\dots\}.
	\]
\end{definition}
Since uniform covering frames decompose the frequency plane into approximately equal subsets, we believe ``Fourier" is an appropriate description of this operator. 

Both $\S_\F$ and $\S_\W$ correspond to neural networks of infinite width and depth, but when used in practice, the networks must be truncated. To truncate $\S_\F$, we keep terms up to a certain depth $K$ and terms belonging to appropriate finite subsets of $\P^k$, for $k=1,2,\dots,K$. In order to define these sets, we first prove the following proposition, which relates the index $p\in\P$ with the location of $\supp(\hat{f_p})$. Observe that (\ref{eq frame}) is a partition of unity statement, but it does not provide any information on how the partitioning is structured. Not surprisingly, the partition of unity has to be done in a ``uniform" way due to the uniform covering property. 

\begin{proposition}
	\label{prop tiling} 
	Let $\F=\{f_0\}\cup\{f_p\colon p\in\P\}$ be a uniform covering frame. There exist a constant $C_1>0$ and subsets $\{\P[m]\subset\P\colon m\geq 1\}$ such that for all integers $m\geq 1$, 
	\begin{equation}
	\label{eq tiling}
	|\hat{f_0}(\xi)|^2+\sum_{p\in\P[m]} |\hat{f_p}(\xi)|^2 =
	\begin{cases}
	\ 1 &\text{if } \xi\in \overline{Q_{C_1m}(0)}, \\
	\ 0 &\text{if } \xi\not\in Q_{C_1(m+1)}(0).
	\end{cases}
	\end{equation}	
	
\end{proposition}

\begin{proof}
	For any set $S\subset\R^d$, let $\diam(S)=\sup_{x,y\in S}|x-y|$ be the diameter of $S$. Define
	\[
	C_1=\max\(\diam(\supp(\hat{f_0})),\ \sup_{p\in\P}\diam(\supp(\hat{f_p}))\).
	\]
	Note that $C_1$ is finite because of the uniform covering property. Indeed, the diameter of $\supp(\hat{f_0})$ is finite since $\hat{f_0}$ is supported in a compact set containing the origin. Fix $R>0$, and by assumption, the closed and connected set $\supp(\hat{f_p})$ can be covered by $N$ cubes of side length $2R$. Then, the diameter of $\supp(\hat{f_p})$ is bounded by $2NR$. 
	
	For integers $m\geq 1$, we define 
	\[
	\P[m]=\{p\in\P\colon \supp(\hat{f_p})\subset \overline{Q_{C_1(m+1)}(0)}\}. 
	\]
	By definition, of $\P[m]$, we have 
	\[
	|\hat{f_0}(\xi)|^2+\sum_{p\in\P[m]} |\hat{f_p}(\xi)|^2 = 0
	\quad \text{if}\quad \xi\not\in Q_{C_1(m+1)}(0).
	\]
	To complete the proof, we prove (\ref{eq tiling}) by contradiction. Suppose there exists $\xi_0\in Q_{C_1m}(0)$ such that  
	\[
	|\hat{f_0}(\xi_0)|^2+\sum_{p\in\P[m]} |\hat{f_p}(\xi_0)|^2 <1. 
	\]
	By the frame condition (\ref{eq frame}), there exists $q\in\P$ such that $|\hat{f_q}(\xi_0)|>0$. Then, $\xi_0\in \supp(\hat{f_q})$ and by definition of $C_1>0$, we have  
	\begin{equation}
	\label{eq sub}
	\supp(\hat{f_q})\subset \overline{Q_{C_1}(\xi_0)}\subset \overline{Q_{C_1(m+1)}(0)}.
	\end{equation}
	For an illustration of this inclusion, see Figure \ref{fig 1}. This shows that $q\in\P[m]$, which contradicts the definition of $\P[m]$.
	
	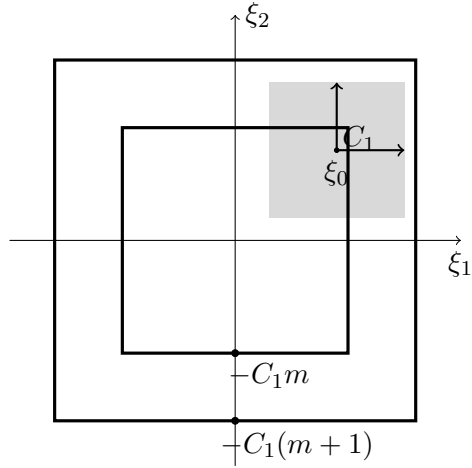
\begin{figure}[h]
		\begin{center}
		\begin{tikzpicture}[scale=0.3]
		\fill [fill,color=black!15!white] (1.5,1) rectangle (7.5,7);
		\draw[fill]
		(0,-5) circle [radius=.15]
		(0,-8) circle [radius=.15]
		(4.5,4) circle [radius=.1];
		\draw[very thick]
		(-5,-5) rectangle (5,5);
		\draw[very thick]
		(-8,-8) rectangle (8,8);
		\draw[->] 
		(-10,0)--(10,0);
		\draw[->]
		(0,-10)--(0,10);
		\draw
		(4.5,4) node[below]{$\xi_0$}
		(1.5,-5) node[below]{$-C_1m$}
		(2.75,-8) node[below]{$-C_1(m+1)$}
		(5.5,5.5) node[below]{$C_1$}
		(10,0) node[below]{$\xi_1$}
		(0,10) node[right]{$\xi_2$};
		\draw[->,thick]
		(4.5,4) -- (7.5,4);
		\draw[->,thick]
		(4.5,4) -- (4.5,7);
		\end{tikzpicture}
		\end{center}

		\caption{An illustration of the inclusions (\ref{eq sub}).}
		\label{fig 1}
	\end{figure} 
\end{proof}

\begin{remark}
	\label{remark sets}
	Suppose $\F$ is a Gabor frame satisfying Proposition \ref{prop gabor} for $A=aI$, where $I$ is the identity transformation on $\R^d$ and $a>0$. By definition, we have $\P=a\Z^d\setminus\{0\}$. We can determine the family of sets $\{\P[m]\colon m\geq 1\}$ satisfying Proposition \ref{prop tiling}. Let $C_1=a$ and
	\[
	\P[m]=\{p\in\P\colon |p|_\infty\leq m\}. 
	\]
	For all integers $m\geq 1$, we have
	\[
	|\hat{f_0}(\xi)|^2+\sum_{p\in\P[m]} |\hat{f_p}(\xi)|^2 =
	\begin{cases}
	\ 1 &\text{if } \xi\in \overline{Q_{am}(0)}, \\
	\ 0 &\text{if } \xi\not\in Q_{a(m+1)}(0).
	\end{cases}
	\]
\end{remark}

From here onwards, let $C_1>0$ be the smallest constant such that Proposition \ref{prop tiling} holds, and let $\{\P[m]\colon m\geq 1\}$ be the family of sets defined in the proposition. Similar to before, we create a tree from this collection of sets. For integers $M,K\geq 1$, we define the discrete set
\begin{equation}
\label{trunP}
\P[M]^K
=\underbrace{\P[M] \times \P[M] \times\cdots\times \P[M]}_{K-\text{times}}
\subset\P^K. 
\end{equation}
Again, we use the convention that $\P[M]^0=\emptyset$. 

\begin{definition}
	The \emph{truncated Fourier scattering transform}, $\S_\F[M,K]\colon L^2(\R^d)\to L^2(\R^d;\ell^2(\Z))$, is defined as
	\begin{equation}
	\label{trunS}
	\S_\F[M,K](f)
	=\{U[p]f *f_0\colon p\in \P[M]^k,\ k=0,1,\dots,K\}. 
	\end{equation}
\end{definition}

\section{Fourier scattering transform}
\label{section FST}

Since $\S_\W$ and $\S_\F$ have the same network structure, it is natural to ask whether $\S_\F$ satisfies the same broad mathematical properties. This is not immediately clear because wavelets and Gabor functions are qualitatively and mathematically different, see \cite{daubechies1990wavelet} for a comparison and discussion on applications. Despite their differences, we prove that $\S_\F$ satisfy all the same properties of $\S_\W$. However, our proof techniques are very different from those of Mallat's. For example, he used scaling and almost orthogonality arguments to exploit the dyadic structure of wavelets, while we use covering and tiling arguments to take advantage of the uniform covering property.

In order to show that $\S_\W$ and $\S_\F$ share the same properties, we need several preliminary results. For all $k\geq 0$ and $p\in\P^k$, it immediately follows from the frame property (\ref{eq frame}) that $U[p]\colon L^2(\R^d)\to L^2(\R^d)$ is bounded with operator norm satisfying $\|U[p]\|_{L^2\to L^2}\leq 1$. The following proposition contains some additional results that follow from the frame property as well. Mallat proved these for his wavelet frame $\W$ in \cite{mallat2012group}, but the arguments only rely on the frame identity (\ref{eq frame}), so we omit their proofs. 

\begin{proposition}
	\label{prop frame} 
	Let $\F=\{f_0\}\cup \{f_p\colon p\in\P\}$ be a uniform covering frame. For $f,g\in L^2(\R^d)$ and integers $K\geq 0$, we have
	\begin{equation}
	\label{eq4}
	\sum_{p\in\P^{K+1}}\|U[p]f\|_{L^2}^2 + \sum_{k=0}^K\sum_{p\in\P^k}\|U[p]f*f_0\|_{L^2}^2
	=\|f\|_{L^2}^2,
	\end{equation}
	and
	\begin{equation}
	\label{eq5}
	\sum_{k=0}^K\sum_{p\in\P^k} \big\|U[p]f*f_0-U[p]g*f_0 \big\|_{L^2}^2 
	\leq \|f-g\|_{L^2}^2.
	\end{equation}	
\end{proposition}

The first identity of Proposition \ref{prop frame} implies that $\S_\F\colon L^2(\R^d)\to L^2(\R^d;\ell^2(\Z))$ is bounded with operator norm satisfying $\|\S_\F\|_{L^2\to L^2\ell^2}\leq 1$. Indeed, we have
\[
\|\S_\F (f)\|_{L^2\ell^2}^2
=\lim_{K\to\infty} \sum_{k=0}^K\sum_{p\in\P^k}\|U[p]f*f_0\|_{L^2}^2
\leq \|f\|_{L^2}^2.
\]

The following proposition is a basic result on positive definite functions. For any integer $k\geq 1$ and dimension $d$, Wendland \cite{wendland1995piecewise} constructed a compactly supported, radial, and positive-definite $C^{2k}(\R^d)$ function. These functions are essentially anti-derivatives of positive-definite polynomial splines. We provide a crude estimate since we do not need all of their smoothness. 

\begin{proposition}
	\label{prop phi}
	There exists a non-negative function $\phi\colon\R^d\to\R$, such that $\hat\phi$ is continuous, decreasing along each Euclidean coordinate, and $\supp(\hat\phi)=\overline{Q_1(0)}$.
\end{proposition}

\begin{proof} 
	For $j=1,2,\dots, d$, let $\phi_j\colon\R\to\R$ be defined by its one-dimensional Fourier transform, 
	\[
	\hat{\phi_j}(\xi_j)=(1-|\xi_j|) \mathds{1}_{[0,1]}(|\xi_j|).
	\]
	Here, $\mathds{1}_S$ is the characteristic function of the set $S$ and for a positive number $x$, $\lfloor x\rfloor$ stands for the integer $n$ satisfying $n\leq x<n+1$. Note that each $\phi_j$ is non-negative because $\hat{\phi_j}$ is a univariate positive-definite function, see \cite{wendland1995piecewise}.	Then, let $\phi\colon\R^d\to\R$ be the function,
	\[
	\phi(x)=\phi_1(x_1)\phi_2(x_2)\cdots\phi_d(x_d).
	\] 
	By construction, $\phi$ is satisfies the desired properties.  
\end{proof}

The following proposition is the crucial exponential decay of energy estimate, and from here onwards, we let $C_0$ denote the constant that appears in the proposition.

\begin{proposition}
	\label{prop decay}
	Let $\F=\{f_0\}\cup\{f_p\colon p\in\P\}$ be a uniform covering frame. There exists a constant $C_0\in (0,1)$ depending only on $\F$, such that for all $f\in L^2(\R^d)$ and integers $K\geq 1$, 
	\[
	C_0^{K-1}\|f*f_0\|_{L^2}^2 + \sum_{p\in\P^K} \|U[p]f\|_{L^2}^2
	\leq C_0^{K-1}\|f\|_{L^2}^2.
	\]
\end{proposition} 

\begin{proof}
	By assumption, $\hat{f_0}$ is continuous, supported in a neighborhood of the origin, and $|\hat{f_0}(0)|=1$. Then, by appropriately scaling the function discussed in Proposition \ref{prop phi}, there exists a non-negative $\phi$ such that $\hat\phi$ is continuous, decreasing along each Euclidean coordinate, $|\hat\phi(0)|>0$, and $|\hat\phi(\xi)|\leq|\hat{f_0}(\xi)|$ for all $\xi\in\R^d$. Then, there exist constants $R=R_\phi>0$ and $C=C_\phi\in (0,1)$, such that $|\hat\phi(\xi)|^2\geq C$ for all $\xi\in Q_R(0)$. By the uniform covering property, there exists an integer $N=N_R>0$, such that for all $p\in\P$, there exist $\{\xi_{p,n}\in\R^d\colon n=1,2,\dots, N\}$ such that
	\[
	\supp(\hat{f_p})
	\subset\bigcup_{n=1}^N Q_R(\xi_{p,n}).
	\]
	
	Let $1\leq k\leq K$ and $q\in\P^{k-1}$. By Plancherel's formula and the above inclusion,
	\[
	\|U[q]f*f_p\|_{L^2}^2
	=\int_{\R^d} |\hat{U[q]f}(\xi)|^2|\hat{f_p}(\xi)|^2\ d\xi
	\leq \sum_{n=1}^N \int_{Q_R(\xi_{p,n})} |\hat{U[q]f}(\xi)|^2|\hat{f_p}(\xi)|^2\ d\xi. 
	\]
	Since $|\hat\phi(\xi-\xi_{p,n})|^2\geq C$ for all $\xi\in Q_R(\xi_{p,n})$, we have
	\[
	\sum_{n=1}^N \int_{Q_R(\xi_{p,n})} |\hat{U[q]f}(\xi)|^2|\hat{f_p}(\xi)|^2\ d\xi
	\leq \frac{1}{C}\sum_{n=1}^N \int_{Q_R(\xi_{p,n})} |\hat{U[q]f}(\xi)|^2|\hat{f_p}(\xi)|^2 |\hat\phi(\xi-\xi_{p,n})|^2\ d\xi.
	\]
	By Plancherel's forumla, we have
	\[
	\sum_{n=1}^N \int_{Q_R(\xi_{n,p})} |\hat{U[q]f}|^2|\hat{f_p}(\xi)|^2 |\hat\phi(\xi-\xi_{p,n})|^2\ d\xi
	\leq \sum_{n=1}^N \|U[q]f*f_p*M_{\xi_{p,n}}\phi\|_{L^2}^2,
	\]
	where $M_y$ is the modulation operator, $M_yf(x)=e^{2\pi iy\cdot x}f(x)$. Using that $\phi\geq 0$ and triangle inequality, we have
	\[
	\sum_{n=1}^N \|U[q]f*f_p*M_{\xi_{p,n}}\phi\|_{L^2}^2
	\leq \sum_{n=1}^N \| \ |U[q]f*f_p|*\phi \|_{L^2}^2.
	\]
	Observe that the terms inside the summation on the right hand side do not depend on the index $n$. Using Plancherel's formula and that $|\hat\phi(\xi)|\leq|\hat{f_0}(\xi)|$ for all $\xi\in\R^d$, we have
	\[
	\| \ |U[q]f*f_p|*\phi \|_{L^2}^2
	\leq \|\ |U[q]f*f_p|*f_0 \|_{L^2}^2.
	\]
	Combining the previous inequalities and rearranging the result, we obtain
	\begin{equation}
	\label{eq1}
	\big\|\ |U[q]f*f_p|*f_0 \|_{L^2}^2
	\geq \frac{C}{N}\|U[q]f*f_p\|_{L^2}^2.
	\end{equation}
	
	The strength of this inequality is that $C$ and $N$ are independent of $q\in\P^{k-1}$ and $p\in\P$. Then, summing this inequality over all $p\in\P$ and $q\in\P^{k-1}$, we see that
	\[
	\sum_{p\in\P^k} \|U[p]f*f_0\|_{L^2}^2
	\geq \frac{C}{N} \sum_{p\in\P^k}\|U[p]f\|_{L^2}^2.
	\]
	Applying the frame identity (\ref{eq frame}) to the left hand side and setting $C_0=1-C/N$, we have
	\[
	\sum_{p\in\P^{k+1}} \|U[p]f\|_{L^2}^2\leq C_0\sum_{p\in\P^k}\|U[p]f\|_{L^2}^2.
	\]
	Iterating the above inequality, we obtain
	\[
	\sum_{p\in\P^K}\|U[p]f\|_{L^2}^2
	\leq C_0^{K-1}\sum_{p\in\P}\|U[p]f\|_{L^2}^2
	= C_0^{K-1}\|f\|^2_{L^2}-C_0^{K-1}\|f*f_0\|_{L^2}^2. 
	\]
\end{proof}

\begin{remark}
	The key step in the proof of Proposition \ref{prop decay} is to obtain the inequality (\ref{eq1}), which is of the form, 
	\[
	\| \ |f*g|*f_0\|_{L^2}\geq C\|f*g\|_{L^2},
	\]
	for some constant $C>0$ \emph{independent} of $f,g\in L^2(\R^d)$. It is straightforward to obtain a lower bound where the constant \emph{depends} on $f$ and $p$. Indeed, suppose $f*g\not=0$. Since $|f*g|$ is continuous, we have
	\[
	(|f*g|)^\wedge(0)
	=\int_{\R^d} |(f*g)(x)|\ dx
	=\|f*g\|_{L^1}
	>0.
	\] 
	By continuity of $|f*g|$, the above inequality, and the assumption that $|\hat{f_0}(0)|=1$, we can find a sufficiently small neighborhood $V=V_{f_0,f,g}$ of the origin and constants $C_{f_0},C_{f,g}>0$, such that $|\hat{f_0}(\xi)|\geq C_{f_0}$ and $|(|f*g|)^\wedge(\xi)|\geq C_{f,g}\|f*g\|_{L^1}$ for all $\xi\in V$. Then,
	\[
	\| \ |f*g|*f_0 \|_{L^2}^2
	\geq \int_{V} |(|f*g|)^\wedge(\xi)|^2 |\hat{f_0}(\xi)|^2\ d\xi
	\geq C_{f_0} C_{f,g}\|f*g\|_{L^1}|V|,
	\]
	where $|V|$ is the Lebesgue measure of $V$. However, both $C_{f,g}$ and $|V|$ depend on $f$ and $g$, and $L^1(\R^d)$ and $L^2(\R^d)$ norms are not equivalent. Thus, this inequality is not very useful for our purposes. However, this naive reasoning suggests that a more sophisticated covering argument, such as the one given in the proof of Proposition \ref{prop decay}, could work. 
\end{remark}

\begin{remark}
	\label{remark energy}
	We have several comments about Proposition \ref{prop decay}.	
	\begin{enumerate}[(a)]\itemsep+.5em
		\item 
		Since $C_0$ describes the rate of decay of $\sum_{p\in\P^k}\|U[p]f\|^2_{L^2}$, it is of interest to determine the optimal (smallest) value $C_0$ for which Proposition \ref{prop decay} holds. Minimizing $C_0$ is equivalent to maximizing the ratio $C/N$, where these constants were defined in the proof. Since $N$ is related to the optimal covering by cubes of side length $2R$ and $C$ is the minimum of $|\hat\phi|^2$ on $Q_R(0)$, both $C$ and $N$ decrease as $R$ increases. 
		
		\item
		Let us examine why the argument proving inequality (\ref{eq1}) fails for the wavelet case. Recall that $\psi_{2^j,r}$ has frequency scale $2^j$ whereas $\varphi_{2^J}$ has frequency scale $2^{-J}$. Using the same argument as in the proof of (\ref{eq1}), we obtain: For each $j>-J$, there exists an integer $N_j>0$ and $C_J\in (0,1)$, such that for all $k\geq 1$, $p\in\Lambda^k$, $r\in G$, and $f\in L^2(\R^d)$,
		\[
		\|\ |U[p]f*\psi_{2^j,r}|*\varphi_{2^J}\|_{L^2}^2
		\geq \frac{C_J}{N_j}\|U[p]f*\psi_{2^j,r}\|_{L^2}^2.
		\]
		The measure of $\supp(\hat{\psi_{2^j,r}})$ is proportional to $2^j$, and $N_j$ is the number of cubes required to cover this set with cubes of side length bounded by a constant multiple of $2^{-J}$. Hence, $\lim_{j\to\infty} N_j=\infty$ and this inequality is not meaningful for large $j$.
		
		\item
		Mallat \cite[Lemma 2.8, page 1344]{mallat2012group} established a qualitative depth decay property for $\S_\W$ and Waldspurger \cite{waldspurger2016exponential} (this preprint is Chapter 5 from her thesis \cite{waldspurger2015these}) proved a quantitative result of a similar nature. Since the completion of the first draft of this paper, the recent preprint of Wiatowski, Grohs, and B\"olcskei \cite{wiatowski2017energy} established an exponential decay of energy property, but requires additional assumptions. We discuss their results in further detail in Section \ref{section discussion}.
		
		\item
		Rather interestingly, numerical experiments in \cite[page 1345]{mallat2012group} have suggested that the exponential decay of energy described in the Proposition \ref{prop decay} holds for the wavelet case. Mallat conjectured that there exists $C\in (0,1)$ such that
		\[
		\sum_{\lambda\in\Lambda^K} \|U[\lambda] f\|_{L^2}^2
		\leq C^{K-1}\|f\|_{L^2}^2,
		\]
		for all $f\in L^2(\R^d)$, and $K\geq 1$. Determining whether this property holds for a given wavelet frame is of interest in the scattering community.
	\end{enumerate}
\end{remark}

We are ready to prove our first main theorem, which shows that $\S_\F$ satisfies several desirable properties as a feature extractor. 

\begin{theorem}
	\label{thm1}
	Let $\F=\{f_0\}\cup\{f_p\colon p\in\P\}$ be a uniform covering frame, and let $\S_\F$ be the Fourier scattering transform. 
	\begin{enumerate}[(a)]\itemsep+1em
		\item 
		\underline{Energy conservation}: For all $f\in L^2(\R^d)$,
		\[
		\|\S_\F(f)\|_{L^2\ell^2}=\|f\|_{L^2}.
		\]
		\item
		\underline{Non-expansiveness}: For all $f,g\in L^2(\R^d)$, 
		\[
		\|\S_\F(f)-\S_\F(g)\|_{L^2\ell^2}\leq \|f-g\|_{L^2}.
		\]
		\item
		\underline{Translation contraction}: There exists $C>0$ depending only on $\F$, such that for all $f\in L^2(\R^d)$ and $y\in \R^d$, 
		\[
		\|\S_\F(T_y f)-\S_\F(f)\|_{L^2\ell^2}
		\leq C|y|\|\nabla f_0\|_{L^1} \|f\|_{L^2}.
		\] 
		
		\item
		\underline{Additive diffeomorphisms contraction}: Let $\epsilon\in [0,1)$ and $R>0$. There exists a universal constant $C>0$, such that for all $(\epsilon,R)$ band-limited $f\in L^2(\R^d)$, and all $\tau\in C^1(\R^d;\R^d)$ with $\|\nabla\tau\|_{L^\infty}\leq 1/(2d)$,
		\[
		\|\S_\F(T_\tau f)-\S_\F(f)\|_{L^2\ell^2}
		\leq C(R\|\tau\|_{L^\infty}+\epsilon)\|f\|_{L^2}.
		\]
	\end{enumerate}
\end{theorem}

\begin{proof}
	\indent
	\begin{enumerate}[(a)]\itemsep+.5em
		\item 
		Using identity (\ref{eq4}) in Proposition \ref{prop frame} and Proposition \ref{prop decay}, we obtain
		\begin{align*}
		\|\S_\F(f)\|_{L^2\ell^2}^2
		&=\lim_{K\to\infty} \sum_{k=0}^K \sum_{p\in\P^k} \|U[p]f*f_0\|_{L^2}^2 \\
		&=\|f\|_{L^2}^2 - \lim_{K\to\infty} \sum_{p\in\P^K}\|U[p]f\|_{L^2}^2
		=\|f\|_{L^2}^2.
		\end{align*}
		\item
		Using the identity (\ref{eq5}) in Proposition \ref{prop frame}, we obtain
		\begin{align*}
		\|\S_\F(f)-\S_\F(g)\|_{L^2\ell^2}^2
		&=\lim_{K\to\infty} \sum_{k=0}^K\sum_{p\in\P^k} \|U[p]f*f_0-U[p]g*f_0\|_{L^2}^2 \\
		&\leq \|f-g\|_{L^2}^2.
		\end{align*}
		
		\item
		By definition, we have
		\[
		\|\S_\F(T_y f)-\S_\F(f)\|_{L^2\ell^2}^2
		=\sum_{k=0}^\infty \sum_{p\in\P^k} \|U[p](T_y f)*f_0-U[p]f*f_0\|_{L^2}^2.
		\]
		Since translation commutes with convolution and the complex modulus, for all $k\geq 0$ and $p\in\P^k$, we have 
		\[
		U[p](T_yf)*f_0
		=T_y(U[p]f)*f_0
		=U[p]f*T_yf_0.
		\]
		This fact, combined with Young's inequality, yields
		\begin{align*}
		\|U[p](T_y f)*f_0-U[p]f*f_0\|_{L^2}
		&=\|U[p]f*(T_y f_0-f_0)\|_{L^2} \\
		&\leq \|T_y f_0-f_0\|_{L^1}\|U[p]f\|_{L^2}.
		\end{align*}
		Then, we have
		\begin{equation}
		\label{eq3}
		\|\S_\F(T_y f)-\S_\F(f)\|_{L^2\ell^2}
		\leq \|T_y f_0-f_0\|_{L^1} \(\sum_{k=0}^\infty\sum_{p\in\P^k} \|U[p] f\|_{L^2}^2\)^{1/2}.
		\end{equation}
		We first bound the summation in (\ref{eq3}). Using Proposition \ref{prop decay}, we have
		\[
		\sum_{k=0}^\infty\sum_{p\in\P^k} \|U[p]f\|_{L^2}^2
		\leq \|f\|_{L^2}^2 + \sum_{k=1}^\infty C_0^{k-1}\|f\|_{L^2}^2
		=\(1+\frac{1}{1-C_0}\)\|f\|_{L^2}^2.
		\]
		To bound the $L^1$ term in (\ref{eq3}), we use the fundamental theorem of calculus, which is justified by the assumption that $f_0\in C^1(\R^d)$. Then, we obtain
		\[
		\int_{\R^d} |f_0(x-y)-f_0(x)|\ dx
		=\int_{\R^d}  \Big| \int_0^1 \nabla f_0(x-ty)\cdot y\ dt \Big|\ dx
		\leq |y|\|\nabla f_0\|_{L^1}.
		\]
		
		\item
		Let $\phi$ be a Schwartz function such that $\hat\phi$ is real-valued, supported in $\overline{Q_2(0)}$, and $\hat\phi(\xi)=1$ for all $\xi\in Q_1(0)$. Let $\phi_R(x)=R^d\phi(Rx)$, and let $f_R=f*\phi_R$. By the non-expansiveness property and triangle inequality
		\begin{equation}
		\label{eq2} 
		\|\S_\F(T_\tau f)-\S_\F(f)\|_{L^2\ell^2}
		\leq \|f-f_R\|_{L^2} + \|T_\tau(f_R)-f_R\|_{L^2}+ \|T_\tau(f_R)-T_\tau f\|_{L^2}.
		\end{equation}
		The bound for the first term in (\ref{eq2}) follows by assumption
		\[
		\|f-f_R\|_{L^2}\leq \epsilon\|f\|_{L^2}.
		\] 
		To bound the second term in (\ref{eq2}), we make the change of variable $u=x-\tau(x)$ and note that
		\[
		\Big|\frac{\partial u}{\partial x}\Big|
		=|\det(I-\nabla\tau(x))|
		\geq (1-d\|\nabla\tau\|_{L^\infty})
		\geq \frac{1}{2},
		\]
		see \cite{brent2015note}. Then, we have
		\begin{align*}
		\|T_\tau(f_R)-T_\tau f\|_{L^2}^2
		&=\int_{\R^d} |f_R(x-\tau(x))-f(x-\tau(x))|^2\ dx \\
		&\leq 2\|f-f_R\|_{L^2}^2 \\
		&\leq 2\epsilon^2\|f\|_{L^2}^2.
		\end{align*}
		It remains to bound the third term of (\ref{eq2}), and we use the argument proved in \cite[Proposition 5]{wiatowski2015mathematical}. We have
		\begin{align*}
		(T_\tau f_R)(x)-f_R(x)
		&=(f*\phi_R)(x-\tau(x))-(f*\phi_R)(x) \\
		&=\int_{\R^d} (\phi_R(x-\tau(x)-y)-\phi_R(x-y)) f(y)\ dy.
		\end{align*}
		The above can be interpreted as an integral kernel operator acting on $f\in L^2(\R^d)$ with kernel
		\[
		k(x,y)=\phi_R(x-\tau(x)-y)-\phi_R(x-y).
		\]
		The proof is completed by verifying that this kernel satisfies the assumptions of Schur's lemma with the appropriate bounds. By the fundamental theorem of calculus, we have 
		\begin{align*}
		|k(x,y)|
		&=\Big|\int_0^1 \nabla\phi_R(x-t\tau(x)-y)\cdot\tau(x)\ dt\Big| \\
		&\leq\|\tau\|_{L^\infty} \int_0^1 |\nabla\phi_R(x-t\tau(x)-y)|\ dt.
		\end{align*}
		\begin{enumerate}[(i)]
			\item 
			For each $x\in\R^d$, we have
			\begin{align*}
			\int_{\R^d} |k(x,y)|\ dy
			&\leq \|\tau\|_{L^\infty} \int_0^1\int_{\R^d} |\nabla\phi_R(x-t\tau(x)-y)|\ dydt \\
			&=R\|\nabla\phi\|_{L^1}\|\tau\|_{L^\infty}.
			\end{align*}
			\item
			For each $y\in\R^d$, we have
			\[
			\int_{\R^d} |k(x,y)|\ dx
			\leq\|\tau\|_{L^\infty} \int_0^1\int_{\R^d} |\nabla\phi_R(x-t\tau(x)-y)|\ dx dt.
			\]
			For fixed $y\in\R^d$ and $t\in[0,1]$, we make the change of variables $v=x-t\tau(x)-y$ and note that
			\[
			\Big|\frac{\partial v}{\partial x}\Big|
			=|\det(I-t\nabla\tau(x))|
			\geq (1-td\|\nabla\tau\|_{L^\infty})
			\geq \frac{1}{2}.
			\]
			Thus, for all $y\in\R^d$,
			\[
			\int_{\R^d} |k(x,y)|\ dx
			\leq 2\|\tau\|_{L^\infty} \int_0^1\int_{\R^d} |\nabla\phi_R(v)| \ dvdt
			=2R\|\nabla\phi\|_{L^1}\|\tau\|_{L^\infty}.
			\]
		\end{enumerate}
		By Schur's lemma, we conclude that 
		\[
		\|T_\tau f-f\|_{L^2}\leq 2R\|\nabla\phi\|_{L^1}\|\tau\|_{L^\infty} \|f\|_{L^2}.
		\]
		Combining the above results, we obtain the inequality
		\[
		\|\S_\F(T_\tau f)-\S_\F(f)\|_{L^2\ell^2}
		\leq (2R\|\nabla\phi\|_{L^1}\|\tau\|_{L^\infty}+\epsilon+\sqrt 2\epsilon)\|f\|_{L^2}.
		\]
		Set $C=\max(2\|\nabla\phi\|_{L^1},1+\sqrt 2)$, which completes the proof.
	\end{enumerate}
\end{proof}

\section{Truncated Fourier scattering transform}
\label{section truncation}

Before we prove anything about the truncated Fourier scattering transform, we provide some motivation for our choice of truncation and the assumptions that we require below. At this point, it is not clear whether $\S_\F[M,K]$ is non-trivial for appropriate $M$ and $K$. For example, let us focus our attention on the first layer. Say we fix $M$ and compute a finite number of first-order coefficients,
\[
\{|f*f_p|*f_0\colon p\in\P[M]\}.
\]
Observe that there exists a non-trivial $f\in L^2(\R^d)$ such that $f*f_p=0$ for all $p\in\P[M]$. This is already problematic, since it shows that this truncated operator has non-trivial kernel and consequently, is not bounded from below. This shows that, in order to truncate just the first layer of coefficients, we need an additional assumption on $f\in L^2(\R^d)$. The most natural assumption is that $f$ is $(\epsilon,R)$ band-limited, and then $M$ can be chosen appropriately depending on $R$. 

Now, we focus our attention on the higher-order terms. The naive idea is to only compute coefficients indexed by the finite subset $\P[M]^k\subset\P^k$, namely,
\[
\{U[p]f*f_0\colon p\in \P[M]^k,\ k=0,1,\dots,K\}.
\]
This truncation tosses away the high frequency terms and might seem reasonable since Proposition \ref{prop decay} showed that the complex modulus pushes higher frequencies to lower frequencies. Indeed, for $f\in L^2(\R^d)$ and $p\in\P$, the proposition showed that $(|f*f_p|)^\wedge$ is non-zero in a neighborhood of the origin even though $(f*f_p)^\wedge$ is compactly supported away from the origin. However, the complex modulus can also push lower frequencies to higher frequencies. To see why, we note that the function $|f*f_p|$ is continuous, but in general, it is not $C^1(\R^d)$; even if we make the very mild assumption that $\nabla f_p\in L^1(\R^d)$, we can only conclude that $|f*f_p|$ has one distributional derivative belonging to $L^2(\R^d)$. Thus, the decay of $(|f*f_p|)^\wedge$ is quite slow, even though $(f*f_p)^\wedge$ is compactly supported! This observation shows that we must be careful when truncating $\S_\F$. However, we shall see that our choice of truncation in fact works, but only requires an alternative argument. 

To demonstrate that our truncation is acceptable, we are concerned with bounding terms of the form, $\|\ |f*f_p| *f_q\|_{L^2}$, where $p\in\P[M]$ and $q\in\P[M]^c$. These are the terms that are thrown away due to truncation, and since $f_p$ has lower frequencies than $f_q$, we expect them to be small. The following proposition shows that this intuition holds. That is, most of the energy is concentrated along the frequency decaying paths.   

\begin{proposition}
	\label{prop path} 
	Let $\F=\{f_0\}\cup \{f_p\colon p\in\P\}$ be a uniform covering frame. For any integer $M\geq 1$, there exists $C_M\in (0,1)$ such that $C_M\to 1$ as $M\to\infty$ and for all integers $k\geq 1$, $p\in\P^k$, and $f\in L^2(\R^d)$,
	\[
	\|U[p]f*f_0\|_{L^2}^2 + \sum_{q\in\P[M]} \|U[p]f*f_q\|_{L^2}^2
	\geq C_M\|U[p]f\|_{L^2}^2.
	\]
\end{proposition}

\begin{proof}
	By Proposition \ref{prop phi}, there exists a non-negative function $\phi$, such that $\hat\phi$ is continuous, decreasing along each Euclidean coordinate, $\supp(\hat\phi)=\overline{Q_1(0)}$, and $|\hat\phi(0)|=1$. Define $\phi_M$ by is Fourier transform, $\hat{\phi_M}(\xi)=\hat\phi(C_1^{-1}M^{-1}\xi)$, for $\xi\in\R^d$. By definition of $\P[M]$, for all $\xi\in\R^d$, 
	\[
	|\hat{\phi_M}(\xi)|^2
	\leq |\hat{f_0}(\xi)|^2+\sum_{q\in\P[M]} |\hat{f_q}(\xi)|^2.
	\]
	Plancherel's formula and this inequality imply
	\begin{align}
	\label{eq6}
	\begin{split}
	\|U[p]f*f_0\|_{L^2}^2+\sum_{q\in\P[M]} \|U[p]f*f_q\|_{L^2}^2
	&\geq \|U[p]f*\phi_M\|_{L^2}^2 \\
	&=\|\ |U[p']f*f_s|*\phi_M\|_{L^2}^2,
	\end{split}
	\end{align}
	where $p=(p',s)$. By definition of $C_1$, there exists $\xi_s\in \R^d$ such that $\supp(\hat{f_s})\subset \overline{Q_{C_1/2}(\xi_s)}$. Applying triangle inequality to the right hand side of (\ref{eq6}) and using that $\phi\geq 0$, we have
	\begin{align}
	\label{eq7}
	\begin{split}
	\|\ |U[p']f*f_s|*\phi_M\|_{L^2}^2
	&\geq \| U[p']f*f_s*M_{\xi_s}\phi_M\|_{L^2}^2 \\
	&=\int_{\R^d} |\hat{U[p']f}(\xi)|^2|\hat{f_s}(\xi)|^2|\hat{\phi_M}(\xi-\xi_s)|^2\ d\xi.
	\end{split}
	\end{align}
	Using that $\hat\phi$ is decreasing along each Euclidean coordinate and the inclusion $\supp(\hat{f_s})\subset \overline{Q_{C_1/2}(\xi_s)}$, we have
	\begin{align*}
	C_M
	&=\inf_{\xi\in\supp(\hat{f_s})} |\hat{\phi_M}(\xi-\xi_s)|^2 \\
	&\geq \inf_{\xi\in \overline{Q_{C_1/2}(0)}} |\hat{\phi_M}(\xi)|^2 \\
	&=|\hat\phi(2^{-1}M^{-1},2^{-1}M^{-1},\dots,2^{-1}M^{-1})|^2
	>0.
	\end{align*}
	Inserting this into (\ref{eq7}) and applying Plancherel's formula yields
	\[
	\int_{\R^d} |\hat{U[p']f}(\xi)|^2|\hat{f_s}(\xi)|^2|\hat{\phi_M}(\xi-\xi_s)|^2\ d\xi
	\geq C_M \|U[p']f*f_s\|_{L^2}^2
	= C_M\|U[p]f\|_{L^2}^2.
	\]
	We have the trivial inequality $C_M\leq 1$, and observe that 
	\[
	\liminf_{M\to\infty} C_M
	\geq \liminf_{M\to\infty} |\hat\phi(2^{-1}M^{-1},2^{-1}M^{-1},\dots,2^{-1}M^{-1})|^2
	=|\hat\phi(0)|^2
	=1. 
	\]
\end{proof}

\begin{remark}
	We have two comments about this proposition.
	\begin{enumerate}[(a)]
		\item 
		This argument fails for wavelet frames. Indeed, we made use of the uniform tiling property in Proposition \ref{prop tiling} and that $\hat{f_p}$ is supported in a cube of side length $C_1$, where $C_1$ is independent of $p\in\P$.  
		\item
		Mallat showed numerically \cite[page 1345]{mallat2012group} that for $\S_\W$, most of the energy is concentrated in frequency decreasing paths. It might be possible to obtain a theoretical result of this nature by adapting his arguments. 
	\end{enumerate}
	
\end{remark}

We are ready to prove our second main theorem, which shows that $\S_\F[M,K]$ is an effective feature extractor. For the stability to diffeomorphisms, we unfortunately require the additional regularity assumption that $f$ is $(\epsilon,R)$ band-limited. We have further comments about this assumption in Section \ref{section discussion}. 

\begin{theorem}
	\label{thm2}
	Let $\F=\{f_0\}\cup\{f_p\colon p\in\P\}$ be a uniform covering frame.  
	\begin{enumerate}[(a)]\itemsep+1em
		
		\item 
		\underline{Upper bound}: For all $f\in L^2(\R^d)$ and integers $M,K\geq 1$,
		\[
		\|\S_\F[M,K](f)\|_{L^2\ell^2}\leq \|f\|_{L^2}.
		\]
		
		\item
		\underline{Lower bound}: Let $\epsilon\in [0,1)$ and $R>0$. There exist integers $K\geq 1$ and $M\geq C_1^{-1}R$ sufficiently large depending on $\epsilon$, such that for all $(\epsilon,R)$ band-limited functions $f\in L^2(\R^d)$,  
		\[
		\|\S_\F[M,K](f)\|_{L^2\ell^2}^2 \geq (C_M^K(1-\epsilon^2)-C_0^{K-1})\|f\|_{L^2}^2.
		\]
		
		\item
		\underline{Non-expansiveness}: For all $f,g\in L^2(\R^d)$ and integers $M,K\geq 1$, 
		\[
		\|\S_\F[M,K](f)-\S_\F[M,K](g)\|_{L^2\ell^2}
		\leq \|f-g\|_{L^2}.
		\]
		
		\item
		\underline{Translation contraction}: There exists a constant $C>0$ depending only on $\F$ such that for all $f\in L^2(\R^d)$, $y\in \R^d$, and integers $M,K\geq 1$, we have
		\[
		\|\S_\F[M,K](T_y f)-\S_\F[M,K](f)\|_{L^2\ell^2}
		\leq C|y|\|\nabla f_0\|_{L^1} \|f\|_{L^2}.
		\] 
		
		\item
		\underline{Additive diffeomorphism contraction}: Let $\epsilon\in [0,1)$ and $R>0$. There exists a universal constant $C>0$, such that for all $(\epsilon,R)$ band-limited $f\in L^2(\R^d)$, and all $\tau\in C^1(\R^d;\R^d)$ with $\|\nabla\tau\|_{L^\infty}\leq 1/(2d)$,
		\[
		\|\S_\F[M,K](T_\tau f)-\S_\F[M,K](f)\|_{L^2\ell^2}
		\leq C(R\|\tau\|_{L^\infty}+\epsilon)\|f\|_{L^2}.
		\]
	\end{enumerate}
\end{theorem}

\begin{proof}
	\indent
	\begin{enumerate}[(a)]\itemsep+0.5em
		
		\item 
		We apply Theorem \ref{thm1} to obtain,
		\[
		\|\S_\F[M,K](f)\|_{L^2\ell^2}
		\leq \|\S_\F(f)\|_{L^2\ell^2}
		=\|f\|_{L^2}.
		\]
		
		\item
		By the tiling property (\ref{eq tiling}), the assumption that $f$ is almost band-limited, and that $C_1M\geq R$, we have
		\[
		\|f\|_{L^2}^2
		=\|f*f_0\|_{L^2}^2+\sum_{p\in\P[M]}\|f*f_p\|_{L^2}^2+\epsilon^2\|f\|_{L^2}^2. 
		\]
		Applying Proposition \ref{prop path} to the summation over $\P[M]$, we obtain,
		\[
		(1-\epsilon^2)\|f\|_{L^2}^2
		\leq \|f*f_0\|_{L^2}^2 + C_M^{-1}\sum_{p\in\P[M]}\|U[p]f*f_0\|_{L^2}^2 + C_M^{-1}\sum_{p\in\P[M]^2}\|U[p]f\|_{L^2}^2.
		\] 
		Continuing to apply Proposition \ref{prop path}, we see that
		\[
		(1-\epsilon^2)\|f\|_{L^2}^2
		\leq \sum_{k=0}^K C_M^{-k} \sum_{p\in\P[M]^k} \|U[p]f*f_0\|_{L^2}^2 + C_M^{-K}\sum_{p\in\P[M]^K} \|U[p]f\|_{L^2}^2.
		\]
		Using that $C_M\in (0,1)$ and Proposition \ref{prop decay}, we have
		\[
		(1-\epsilon^2)\|f\|_{L^2}^2
		\leq C_M^{-K}\sum_{k=0}^K \sum_{p\in\P[M]^k} \|U[p]f*f_0\|_{L^2}^2 + C_M^{-K} C_0^{K-1} \|f\|_{L^2}^2.
		\]
		Rearranging, we obtain
		\[
		\|\S_\F[M,K](f)\|_{L^2\ell^2}^2
		=\sum_{k=0}^K \sum_{p\in\P[M]^k} \|U[p]f*f_0\|_{L^2}^2
		\geq (C_M^K(1-\epsilon^2)-C_0^{K-1})\|f\|_{L^2}^2.
		\]
		Since $C_M\to 1$ as $M\to\infty$ and $C_0\in (0,1)$ independent of $M$, for fixed $\epsilon$, we can pick $K$ and $M$ sufficiently large so that $C_M^K(1-\epsilon^2)-C_0^{K-1}>0$. Note that this term represents the error due to approximating $f$ by a band-limited function, the horizontal truncation, and the depth truncation.

		\item[(c)-(e)]
		For any $f,g\in L^2(\R^d)$, we have
		\[
		\|\S_\F[M,K](f)-\S_\F[M,K](g)\|_{L^2\ell^2}^2
		\leq \|\S_\F(f)-\S_\F(g)\|_{L^2\ell^2}^2.
		\]
		Applying Theorem \ref{thm1} completes the proof. 
	\end{enumerate}	
\end{proof}

\section{A motivational experiment}
\label{section motivation}

We propose a simple algorithm that computes $\S_\F[M,K](f)$ for an input $f$, and we call this algorithm the \emph{fast Fourier scattering transform}. 

\begin{algorithm}
	Fast Fourier scattering transform. \\
	\indent \textbf{Input:} Function $f$, network depth $K\geq 1$, network width $M\geq 1$ \\
	\indent \textbf{Construct:} Frame elements $\{f_0\}\cup\{f_p\colon p\in\P[M]\}$ \\
	\indent \textbf{for} $k = 1,2,\dots,K$ \\
	\indent\indent \textbf{for each} $p=(p',p_k)\in\P[M]^k$ \\
	\indent\indent\indent Compute $U[p]f=U[(p',p_k)]f = | U[p']f*f_{p_k}|$ and $U[p]f*f_0$ \\
	\indent\indent \textbf{end} \\
	\indent \textbf{end}
\end{algorithm}

\begin{remark}
	The reason we call this algorithm fast is because Theorem \ref{thm2} quantifies the amount of information lost due to truncation, so we do not have to calculate an enormous number of scattering coefficients. In fact, it is possible to make the algorithm even faster:
	\begin{enumerate}[(a)]\itemsep+.5em
		\item
		In many applications, the input function $f$ is real-valued. Suppose that for each $p\in\P[M]$, $\hat{f_p}$ is real-valued and there exists a unique $q\in\P[M]$ such that $\hat{f_p}(\xi)=\hat{f_q}(-\xi)$ for all $\xi\in\R^d$. Then, we have $|f*f_p|=|f*f_q|$. For any $k\geq 1$ and $p\in\P^k$, $U[p]f$ is also real-valued, so the same reasoning shows that for all $k\geq 1$ and $p\in\P[M]^k$, there exists a unique $q\in\P[M]^k$ such that $U[p]f=U[q]f$. Thus, we only need to compute half of the coefficients in the fast Fourier scattering transform. These assumptions hold, for example, if $\F$ is a Gabor frame defined in Proposition \ref{prop gabor} and the window $\hat{f_0}=\hat g$ is real-valued and symmetric about the origin.
		
		\item
		The proof of Proposition \ref{prop path} shows that coefficients of the form $|\ |f*f_p|*f_q|$ have small $L^2(\R^d)$ norm whenever $p\in\P[M]$, $q\in\P[N]$, and $M\ll N$. Since our algorithm still computes such coefficients, its runtime can be greatly reduced by not computing these coefficients. 
	\end{enumerate}
\end{remark}

Figure \ref{fig 3} is an example that compares the features generated by $\S_\F$ and $\S_\W$. The Matlab code that reproduces this figure, as well as Fourier scattering software, can be downloaded at \cite{li2017personal}. We use the publicly available ScatNet toolbox \cite{mallat2017scattering} to produce the analogous wavelet scattering features.

\begin{figure}[!]
	\hspace*{-8cm} \includegraphics[scale=0.6]{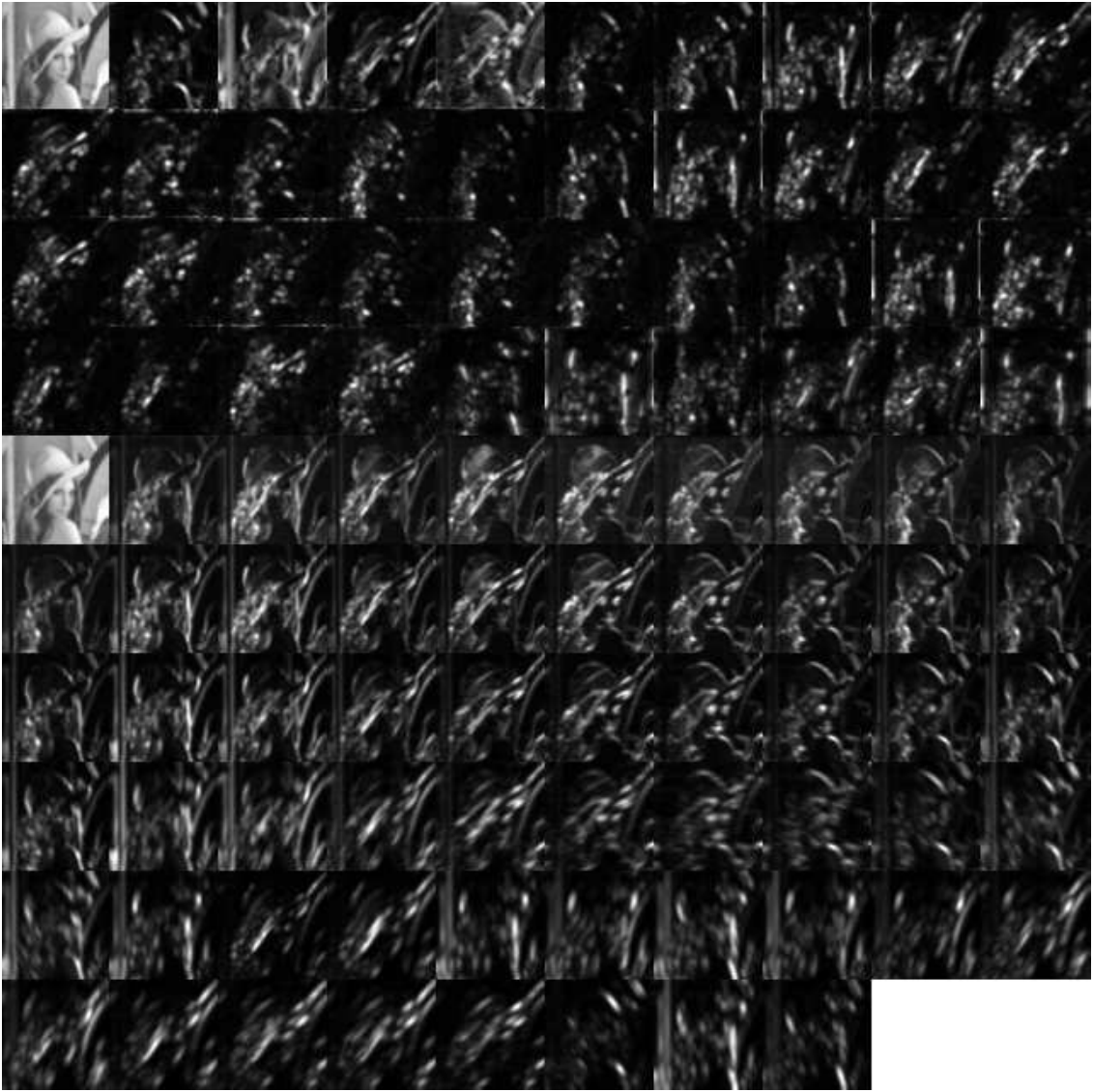}
	\caption{Fourier and windowed scattering transform coefficients whose norms exceed 0.5\% and ordered by depth. The top four rows display $(1,33,6)$ zero, first, and second order Fourier scattering transform coefficients. Bottom six rows display $(1,40,17)$ zero, first, and second order windowed scattering coefficients.}
	\label{fig 3}
\end{figure}

For the comparison, we use the standard $512\times 512$ Lena image as our model example. Due to its varied content, Lena image allows to emphasize the differences between $\S_\F$ and $\S_\W$. We interpret the square image as samples of a function $f$, namely
\[
\{f(m_1,m_2)\colon m_1=1,2,\dots,512,\ m_2=1,2,\dots,512\}. 
\]
In order to make a fair comparison, we chose appropriate parameters. $\S_\F[M,K]$ requires specification of $\F$ (for simplicity, we use a Gabor frame satisfying Proposition \ref{prop gabor} with parameter $a$ free to chose, $A=aI$, and real-valued even window $\hat{f_0}=\hat g$), the integer $M$ that controls the cardinality of $\P[M]$, and the depth of the network $K$. On the other hand, $\S_\W$ requires specification of the number of angles in the rotation group $G$, the range of dyadic scales (controlled by the coarsest scale $J$), and the depth of the network. It is appropriate to pick $a$ and $J$ such that $\hat{f_0}$ and $\hat{\varphi_{2^J}}$ are approximately supported in a ball of the same size. By this choice, the zero-th order coefficients of both transformations have approximately the same resolution. Next, we chose $M$ and $G$ so that both transforms have the same number of first-order coefficients. For both, we use a network with 2 layers. With these considerations, we computed $(1,40,560)$ and $(1,40,600)$ zero, first, and second order coefficients for $\S_\F$ and $\S_\W$, respectively. 

While it is interesting to look at each coefficient, there are obviously too many to display. Hence, Figures \ref{fig 3} illustrates the Fourier and wavelet scattering coefficients with $L^2(\R^d)$ norm greater than $0.005*\|f\|_{L^2}$. In order to clearly display the coefficients, each one is normalized to have a maximum of 1 (but only when plotted). The norm of the coefficients decrease as a function of the depth, so we chose $0.005$ as the threshold parameter in order to display several of the second-order coefficients. While we only show the largest ones, it does not necessarily mean that they are the most important since it is plausible that smaller coefficients contain informative features. 

We first concentrate on the Fourier scattering features. As seen in the Figure \ref{fig 3}, the first-order coefficients extract distinct features of the image, and in particular, extract the most prominent edges in image. While the first-order coefficients capture individual features -- various components of the hat, her hair, the feature, the background, and her facial features -- the second-order coefficients appear to capture a combination of features. In general, it is difficult to substantiate what functions of the form $|\ |f*f_p|*f_q|$ intuitively mean. This is partly because the Fourier transform is the standard tool for analyzing convolutions, but is not well suited for handling non-linear operators such as the complex modulus. For the wavelet case, Mallat has heuristically argued that coefficients of the form $|\ |f*\psi_{2^j,r}|*\psi_{2^k,s}|$ describe interactions between scales $2^{-j}$ and $2^{-k}$ \cite{mallat2016understanding}. By the same reasoning, coefficients of the form $|\ |f*f_p|*f_q|$ describe the interactions between oscillations arising from the uniform Fourier scales $|p|_\infty$ and $|q|_\infty$. 

Let us compare the features generated by $\S_\F$ and $\S_\W$. Their zero-order coefficients are almost identical due to our parameter choices; recall that we chose $f_0$ and $\varphi_{2^J}$ to have approximately the same resolution. In some sense, the first-order coefficients illustrate the main differences between wavelet and Gabor transforms. Indeed, wavelets are multi-scale representations, so the first-order coefficients of $\S_\W$ have higher resolution than those of $\S_\F$, and the latter are of a fixed low resolution. 

By inspection, the second-order coefficients of $\S_\F$ and $\S_\W$ appear to capture significantly different features. For example, the second-order Fourier scattering coefficients mainly capture the oscillatory pattern of the feather, whereas the second-order windowed scattering coefficients focused more on her hair and on the round shapes in the image. Again, it is hard to precisely describe what information the second-order terms capture. Finally, $\S_\F$ has a fewer number of second-order coefficients that exceed the threshold parameter, which is consistent with our theory that $\S_\F$ satisfies an exponential decay of energy property.

\section{Discussion}
\label{section discussion}

In addition to the papers from Mallat's group \cite{mallat2012group, bruna2013invariant,waldspurger2016exponential}, we note that Wiatowski and B\"{o}lcskei \cite{wiatowski2015mathematical} also constructed a scattering-like transform, which they called the \emph{generalized feature extractor} $\Phi$. For other examples of combining artificial neural networks with harmonic analysis, see \cite{candes1999harmonic, shaham2016provable} and the references therein. In this section, we compare our results with the aforementioned papers on scattering, in the following ways. 

\begin{enumerate}[(a)] \itemsep+.5em
	
	\item 
	\underline{Generality and flexibility}. We used Gabor frames as the model example, but our theory applies to any uniform covering frame. This is an important point because elements of Gabor and wavelet frames are algebraically related; in contrast, the learned filters in convolutional neural networks are independent of each other and typically do not satisfy such rigid relationships \cite{lee2008sparse}. 
	
	So far, the theoretical papers of Mallat and collaborators have exclusively focused on wavelet scattering transforms. 
	
	Wiatowski and B\"{o}lcskei \cite{wiatowski2015mathematical} studied a scattering framework that is more general than ours and Mallat's. Instead of using the same semi-discrete wavelet frame for each layer of the network, they used (not necessarily tight) semi-discrete frames, and allowed each layer of the transformation to use a different frame. Their theory allowed for a variety of non-linearities at each layer, including the complex modulus, and allowed sub-sampling to be incorporated into each layer.
	
	\item
	\underline{Energy conservation}. We showed that $\S_\F$ is energy conserving for any $f\in L^2(\R^d)$, which was a consequence of the exponential decay of energy property, Proposition \ref{prop decay}.	
	
	Mallat showed that $\S_\W$ is energy preserving but that result required a restrictive and technical \emph{admissibility condition} on $\psi$, see \cite[pages 1342-1343]{mallat2012group}. This is completely different from the usual admissibility condition related to the invertibility of the continuous wavelet transform. We cannot offer an intuitive explanation for what Mallat's complicated admissibility condition means. Numerical calculations have supported the assertion that an analytic cubic spline Battle-Lemari\'{e} wavelet is admissible for $d=1$, \cite[page 1345]{mallat2012group}. To our best knowledge, it is currently unknown if other Littlewood-Paley wavelets, such as curvelets \cite{candes2004new} or shearlets  \cite{guo2006sparse,guo2007optimally,czaja2014anisotropic}, are admissible.
	
	Mallat's argument for energy conservation is qualitative and cannot be used to deduce a quantitative result because he approximated $f\in L^2(\R^d)$ with a function in the logarithmic Sobolev space. Motivated by this observation, Waldspurger \cite[Theorem 3.1]{waldspurger2016exponential} gave mild assumptions on the generating wavelet $\psi$, see the reference for the explicit hypotheses. Under these assumptions, the following holds: there exists $r>0$ and $a>1$ such that for any integer $k\geq 2$ and real-valued $f\in L^2(\R)$, 
	\begin{equation}
	\label{eq8}
	\sum_{\lambda\in\Lambda^k} \|U[\lambda]f\|_{L^2}^2
	\leq \int_{\R} |\hat f(\xi)| ^2 \(1-\exp\(-\frac{2\xi^2}{r^2 a^{2k}}\)\)\ d\xi. 
	\end{equation}
	This inequality quantifies the intuition that wavelet scattering coefficients become progressively concentrated in lower frequency regions. For a $\psi$ satisfying these assumptions, she showed that the resulting scattering transform conserves energy. However, her result only applies to one-dimensional real-valued functions, but it is possible that they can be adapted to more general situations. In particular, they do not apply to curvelets and shearlets.
	
	In general, $\Phi$ is not energy preserving and possibly has trivial kernel. This is not surprising, because the lower bounds on $\|\S_\F(f)\|_{L^2\ell^2}$ and $\|\S_\W (f)\|_{L^2\ell^2}$ are related to the amount of energy the complex modulus pushes from high to low frequencies from one layer to the next. Thus, it would be surprising if any kind of frame and any type of nonlinearity has the same kind of effect.
	
	\item 
	\underline{Non-expansiveness}. The non-expansiveness property holds for $\S_\F$, $\S_\W$, and $\Phi$ because this is a consequence of the frame property and network structure.
	
	\item
	\underline{Translation contraction estimate}. Wiatowski and B\"{o}lcskei did not provide a translation estimate for $\Phi$. 
	
	Our translation estimate, Theorem \ref{thm1}c, is similar to Mallat's translation estimate \cite[Theorem 2.10]{mallat2012group}: There exists a constant $C>0$ such that for all $f\in L^2(\R^d)$ and $y\in\R^d$, 
	\[
	\|\S_\W(T_yf)-\S_\W(f)\|_{L^2\ell^2}
	\leq C2^{-J}|y|\(\sum_{k=0}^\infty \sum_{\lambda\in\Lambda^k} \|U[\lambda] f\|_{L^2}^2\).
	\]
	To see why, our $\|\nabla f_0\|_{L^1}$ plays the same role as his $C 2^{-J}$ because if $f_0(x)=2^{-dJ}\phi(2^{-J}x)$ for some smooth $\phi\in L^1(\R^d)$, like in Mallat's case, then 
	\[
	\|\nabla f_0\|_{L^1} 
	=\|\nabla \phi\|_{L^1}2^{-J}
	=C2^{-J}.
	\]
	The only difference is that our inequality is more transparent because it depends on $\|f\|_{L^2}$, whereas Mallat's estimate depends on the more complicated term $\sum_{k=0}^\infty \sum_{\lambda\in\Lambda^k}\|U[\lambda]f\|_{L^2}^2$. This term is finite if $f$ belongs to a certain logarithmic Sobolev space and $\psi$ satisfies the admissibility condition.
	
	\item
	\underline{Diffeomorphism contraction estimate}. Our diffeomorphism estimate, Theorem \ref{thm1}d, is essentially identical to the corresponding estimate for $\Phi$ \cite[Theorem 1]{wiatowski2015mathematical}. 
	
	Mallat's diffeomorphism estimate \cite[Theorem 2.12]{mallat2012group} is quite different. It says, for any $\tau\in C^2(\R^d;\R^d)$ with $\|\nabla\tau\|_{L^\infty}$ sufficiently small, there exists $C(J,\tau)>0$ such that for all $f\in L^2(\R^d)$,
	\begin{equation}
	\label{eq9}
	\|\S_\W(T_\tau f)-\S_\W(f)\|_{L^2\ell^2}
	\leq C(J,\tau) \sum_{k=0}^\infty \sum_{\lambda\in\Lambda^k}\|U[\lambda]f\|_{L^2}.
	\end{equation}
	We caution that this estimate is only meaningful if $\sum_{k=0}^\infty \sum_{\lambda\in\Lambda^k} \|U[\lambda]f\|_{L^2}<\infty$; we expect that characterizing this class of functions is a difficult task. A sufficient (but perhaps not necessary) condition for this term being finite is that $f$ belongs to a logarithmic Sobolev space and $\psi$ satisfies the admissibility condition, and in which case, we have
	\[
	\sum_{k=0}^\infty \sum_{\lambda\in\Lambda^k} \|U[\lambda]f\|_{L^2}
	\leq \sum_{k=0}^\infty \sum_{\lambda\in\Lambda^k} \|U[\lambda]f\|_{L^2}^2 
	<\infty.
	\]
	For example, a function belongs to this logarithmic Sobolev space if its average modulus of continuity is bounded, and this condition is much weaker than band-limited. The point here is that, the currently known results for both $\S_\F$ and $\S_\W$ require additional regularity assumptions on $f$ in order to establish stability to diffeomorphisms, and removing all regularity assumptions appears to be challenging.
	
	On the other hand, the inequality (\ref{eq9}) has applications to finite depth wavelet scattering networks. Indeed, if one considers a transform that only includes $K$ layers, then the summation on the right hand side terminates at $k=K$. Then we have
	\[
	C(J,\tau) \sum_{k=0}^K \sum_{\lambda\in\Lambda^k} \|U[\lambda]f\|_{L^2}
	\leq C(J,\tau) (K+1)\|f\|_{L^2},
	\]
	and this inequality holds without additional regularity assumptions on $f$. As we have already mentioned, it is of interest if one could upper bound the left hand side independent of $K$. 
	
	\item
	\underline{Rotational invariance}. By exploiting that $\W$ is partially generated by a finite rotation group $G$, Mallat defined a variant of the windowed scattering operator $\tilde\S_\W$ that is $G$-invariant: $\tilde\S_\W(f\circ r)=\tilde\S_\W(f)$ for all $f\in L^2(\R^d)$ and $r\in G$, see \cite[Section 5]{mallat2012group}. In the companion paper \cite{czaja2016uniform}, we construct uniform covering frames that are partially generated by $G$, and define the $G$-invariant rotational Fourier scattering transform $\tilde\S_\F$. In general, it is not possible to modify $\Phi$ in order to obtain a $G$-invariant $\tilde\Phi$, because the underlying frame elements need to be ``compatible" with the action of $G$.  
	
	\item
	\underline{Finite scattering networks}. The main advantage of the uniform covering frame approach is that we have a theory for appropriate truncations of $\S_\F$, and in particular, the lower bound in Theorem \ref{thm2}. Of course, the lower bound requires additional regularity assumptions, but in view of the discussion in Section \ref{section truncation}, some kind of assumption is necessary.
	
	In contrast, it is not known whether the analogous result holds for a finite width and depth truncation of $\S_\W$ because there is no known quantitative analogues of Propositions \ref{prop decay} and \ref{prop path} for the wavelet case. It might be possible that Waldspurger's quantitative estimate (\ref{eq8}) can be used to obtain something similar to Proposition \ref{prop decay}, but as we remarked earlier, her results only hold in one dimension. After the completion of the first draft of this paper, Wiatowski, Grohs, and B{\"o}lcskei \cite{wiatowski2017energy} used ideas similar to Waldspurger's to prove the exponential decay of energy property for the wavelet case, but their result only holds in dimension $d=1$ and requires several restrictive assumptions including Sobolev regularity of $f$. While their results have applications to finite depth networks, they do not address finite width networks. 
	
	Finally, there is no analogous energy decay result for the generalized feature extractor $\Phi$. It would be surprising if it were possible to prove that property without additional assumptions on the underlying frame. 
\end{enumerate}

\section{Acknowledgements} 
This work was supported in part by the Defense Threat Reduction Agency grant HDTRA 1-13-1-0015 and by Army Research Office grant W911 NF-16-1-0008. The authors thank the anonymous reviewers for their helpful comments and feedback. The second named author thanks fellow graduate students, Chae Clark, Yiran Li, and Mike Pekala, for helpful discussions.

\bibliography{TFscatteringbib}
\bibliographystyle{plain}

\end{document}